\documentclass[a4paper,11pt]{article}

\usepackage{amsthm}
\usepackage{amsfonts,amssymb,amsmath}
\usepackage[dvips]{graphicx}
\usepackage{hyperref}

\newtheorem{theorem}{\bf Theorem}

\newtheorem{lemma}{\bf Lemma}
\newtheorem{remark}{\bf Remark}

\theoremstyle{definition}

\usepackage[margin=1in]{geometry}

\newcommand\be{\begin{eqnarray*}}
\newcommand\ee{\end{eqnarray*}}
\newcommand\beq{\begin{equation}}
\newcommand\eeq{\end{equation}}

\newcommand\ben{\begin{eqnarray}}
\newcommand\een{\end{eqnarray}}

\usepackage{xcolor}
\newcommand*{\missingreference}{\colorbox{red}{?reference?}}
\newcommand*{\missingcitation}{\colorbox{red}{?citation?}}
\makeatletter
\def\@setref#1#2#3{%
  \ifx#1\relax
   \protect\G@refundefinedtrue
   \nfss@text{\reset@font\missingreference}%
   \@latex@warning{Reference `#3' on page \thepage \space
             undefined}%
  \else
   \expandafter#2#1\null
  \fi}
\def\@citex[#1]#2{\leavevmode
  \let\@citea\@empty
  \@cite{\@for\@citeb:=#2\do
    {\@citea\def\@citea{,\penalty\@m\ }%
     \edef\@citeb{\expandafter\@firstofone\@citeb\@empty}%
     \if@filesw\immediate\write\@auxout{\string\citation{\@citeb}}\fi
     \@ifundefined{b@\@citeb}{\hbox{\reset@font\missingcitation}%
       \G@refundefinedtrue
       \@latex@warning
         {Citation `\@citeb' on page \thepage \space undefined}}%
       {\@cite@ofmt{\csname b@\@citeb\endcsname}}}}{#1}}
\makeatother

\makeatletter
\def\blfootnote{\gdef\@thefnmark{}\@footnotetext}
\makeatother

\usepackage[final]{showkeys} 
\usepackage{cleveref}

\usepackage[final]{pdfpages}

\usepackage[centerlast]{caption}
\usepackage{float}
\usepackage{graphicx}
\usepackage{placeins}
\usepackage{tikz}

\date{\today}
\title{Convex sequences may have thin additive bases}

\author{Imre Z. Ruzsa\\
\small Alfr\'{e}d R\'{e}nyi Institute of Mathematics\\[-0.8ex]
\small Hungarian Academy of Sciences\\[-0.8ex] 
\small Budapest, Hungary \\
\small\tt ruzsa.z.imre@renyi.mta.hu\\
\and
Dmitrii Zhelezov\\
\small Alfr\'{e}d R\'{e}nyi Institute of Mathematics\\[-0.8ex]
\small Hungarian Academy of Sciences\\[-0.8ex] 
\small Budapest, Hungary \\
\small \tt dzhelezov@gmail.com
}

\begin{document}

\blfootnote{\textup{2000} \textit{Mathematics Subject Classification}: 11B13}

\maketitle

\begin{abstract}
	For a fixed $c > 0$ we construct an arbitrarily large set $B$ of size $n$ such that its sum set $B+B$ contains a convex sequence of size $cn^2$, answering a question of Hegarty.
\end{abstract}

\section*{Notation}
The following notation is used throughout the paper. The expressions $X \gg Y$, $Y \ll X$, $Y = O(X)$, $X = \Omega(Y)$ all have the same meaning that there is an absolute constant $c$ such that $|Y| \leq c|X|$. 

 If $X$ is a set then $|X|$ denotes its cardinality. 

For sets of numbers $A$ and $B$ the \emph{sumset} $A + B$ is the set of all pairwise sums  
$$
\{ a + b: a \in A, b \in B \}.
$$

\section{Introduction}

Let $A = \{a_i\}, i = 1\ldots n$ be a set\footnote{Sometimes we use the word \emph{sequence} to emphasize the ordering. } of real numbers ordered in a way that $a_1 \leq a_2  \leq \ldots \leq a_n$.  Recall that $A$ is called \emph{convex} if the gaps between consecutive elements of $A$ are strictly increasing, that is 
$$
    a_2 - a_1 < a_3 - a_2 < \ldots < a_n - a_{n-1}.
$$
Studies of convex sets were initiated by Erd\H{o}s who conjectured that any convex set must grow with respect to addition, so that the size of the set of sums $A+A := \{a_1 + a_2 : a_1, a_2 \in A \}$ is significantly larger than the size of $A$.

The first non-trivial bound confirming the conjecture of   Erd\H{o}s was obtained by Hegyv{\'a}ri \cite{MR858397}, and the state of the art bound is due to Schoen and Shkredov \cite{MR2825592}, who proved that for an arbitrary convex set $A$ holds
$$
	|A+A| \geq C|A|^{14/9}\log^{-2/3} |A|
$$
for some absolute constants $C, c > 0$. It is conjectured that in fact 
$$
	|A+A| \geq C(\epsilon)|A|^{2-\epsilon}
$$
holds for any $\epsilon > 0$.

In general, it is believed that convex sets cannot be additively structured. In particular, a few years ago Hegarty asked\footnote{The original MathOverflow question is contrapositive to our reformulation which is technically slightly more convenient to state. } \cite{HegartyQuestion} whether there is a constant $c>0$ with the property 
that there is a set $B$ of arbitrarily large size $n$ such that $B+B$ contains a convex set of size $cn^2$. 

Recall that $B$ is a \emph{basis} (of order two) for a set $A$ if $A \subset B + B$. In other words, Hegarty asked if a convex set of size $n$ can have a thin additive basis (of order two) of size as small as $O(n^{1/2})$, which is clearly the smallest possible size up to a constant.  

Perhaps contrary to the intuition that convex sets lack additive structure, we present a construction which answers Hegarty's question in the affirmative. Our main result is as follows.

\begin{theorem} \label{thm:main}
 There is $c>0$ such that for any $m$ there is a set $B$ of size $n  > m$ such that $B+B$ contains a convex set of size $cn^2$. 
\end{theorem}

\section{Construction}

Assume $n$ is fixed and large. We will construct a set $B$ of size $O(n)$ such that $B+B$ contains a convex set of size $\Omega(n^2)$. Theorem \ref{thm:main} will clearly follow.

The following constants (we assume $n$ is fixed) will be used throughout the proof.
$$
	\alpha := \frac{1}{n^2} \,\,\,\,\,
	\gamma := \frac{1}{1000n^3} \,\,\,\,\,
	\epsilon := 0.1
$$

Define 
\be
	x_i &=& i + (\alpha + \gamma)i^2   \\
	y_j &=& j  - \alpha j^2.
\ee
Next, we define
$$
	B_k = \{x_i + y_j : i + j = k \},
$$
where $i$ and $j$ are allowed to be negative. 

Let $k \in [.999n, n]$  so that $\alpha k^2 \in [.99, 1]$. For such an integer $k$ writing $j = k - i$ we have that the $i$th element of $B_k$ is given by
\beq \label{eq:block}
 	b^{(k)}_i = k + (\alpha + \gamma)i^2 - \alpha(k-i)^2  = (k - \alpha k^2) + \gamma i^2 + 2 i k \alpha.
\eeq
Now assume that $i$ ranges in $[-n, 2n]$. The consecutive differences  $b^{(k)}_{i+1} - b^{(k)}_i$ are then given by
$$
	\Delta^{(k)}_i := \gamma(2i+1) + 2k\alpha.
$$ 
Observe that $\Delta^{(k)}_i$ are positive and increasing, thus the block $B_k := \{ b^{(k)}_i \}^{2n}_{-n}$ is convex. Further, by (\ref{eq:block}) for sufficiently large $n$ we have
\ben
	b^{(k)}_{-n} &=& k - \alpha k^2 + \gamma n^2 - 2nk \alpha \in [k - 2.9, k - 3]  \label{eq:lowerbound} \\
	b^{(k)}_{2n} &=& k - \alpha k^2 + \gamma (2n)^2 + 4nk \alpha \in [k+2.9, k + 3.1],  \label{eq:upperbound} 
\een
so $B_k \subset [k-3, k+3] + [-\epsilon, \epsilon]$.

Now we are going to a build large convex sequence out of blocks $B_k$ with $4 | k$. Since each $B_k$ is already convex, it remains to show how to glue together $B_k$ and $B_{k+4}$ so that the resulting set is again convex. We proceed with the following simple lemma.
\begin{lemma} \label{lm:glueing}
	Let $X = \{x_i \}^N_{i=0}$ and $Y = \{ y_j \}^M_{j=0}$ be two convex sequences and there are indices $u$ and $v$ such that 
	$$
	 [x_u, x_{u+1}]	\subset [y_v, y_{v+1}]. 
	$$
	Then 
	$$
		Z := 	\{x_i \}^u_{i=0} \cup \{ y_j \}^M_{j=v+1}
	$$
	is a convex sequence.
\end{lemma}
\begin{proof}
	Since  $[x_u, x_{u+1}]	\subset [y_v, y_{v+1}]$ we have that 
	$$
	x_{u} - x_{u-1} < x_{u+1} - x_u < y_{v+1} - x_u.
	$$ 
	On the other hand, 
	$$
	y_{v+1} - x_u < y_{v+1} - y_v < y_{v+2} - y_{v+1}.
	$$
\end{proof}

By Lemma \ref{lm:glueing}, in order to merge $B_k$ and $B_{k+4}$ it suffices to find two consecutive  elements $b^{(k)}_i, b^{(k)}_{i+1} \in B_k$ in between two consecutive elements $b^{(k+4)}_j, b^{(k+4)}_{j+1} \in B_{k+4}$. 
Define
\be
	\delta &:=& \max_{i \in [-n, 2n]}   \Delta^{(k)}_i   \\
	\Delta &:=& \min_{i \in [-n, 2n]}   \Delta^{(k+4)}_i 
\ee
We have 
\ben 
	\delta < 3n\gamma + 2k\alpha < \frac{2.1}{n}   \label{eq:delta1} \\
	\Delta - \delta > 8\alpha - 10n\gamma > \frac{6}{n^2}  \label{eq:delta2}.  
\een

Let $b^{(k)}_v$ be the least element in $B_k$ greater than $b^{(k+4)}_{-n}$ (such element exists by (\ref{eq:upperbound})). We claim that with $m := \lceil n/2 \rceil + 1$ holds
$b^{(k+4)}_{-n + m} > b^{(k)}_{v+m}$, which in turn by the pigeonhole principle guarantees the arrangement of elements required by Lemma \ref{lm:glueing}. 

Indeed, by our choice of $v$ 
\beq \label{eq:startdist}
0 \leq d : =b^{(k)}_v - b^{(k+4)}_{-n} \leq \delta
\eeq
But by (\ref{eq:delta1}), (\ref{eq:delta2})
\beq
	b^{(k+4)}_{-n+m} - b^{(k)}_{v+m} > -d + m(\Delta-\delta) > \frac{3}{n} - \delta > 0,
\eeq
so the claim follows.

It remains to note that 
$$
b^{(k)}_{v+m} < b^{(k+4)}_{-n} + m\Delta < (k+1+\epsilon) +  \frac{2n^2\alpha}{2} + 4\gamma n < k + 2.2
$$ 
and thus $v + m < 2n$ by (\ref{eq:upperbound}).  This verifies that $b^{(k)}_v, b^{(k)}_{v+m} \in B_k$.


\section{Putting everything together}

Applying the procedure described in the previous section, we can glue together consecutive blocks $B_{4l}$ with $4l: = k \in [0.999n, n]$.  Let $A$ be the resulting convex sequence.
First, observe there are $\Omega(n)$ blocks being merged. Moreover, each interval $[4l-1+\epsilon, 4l +1 -\epsilon]$ is covered only by the block $B_{4l}$ and by (\ref{eq:lowerbound}), (\ref{eq:upperbound}), (\ref{eq:delta1}) contains $\Omega(n)$ elements from $B_{4l}$, so $|A| = \Omega(n^2)$. 
On the other hand, by our construction, $A$ is contained in the sumset $B+B$ of $B := \{ x_i\}^{2n}_{-2n} \cup \{ y_j\}^{2n}_{-2n}$ of size $O(n)$.

\begin{remark}
It follows from our construction that there are arbitrarily large convex sets $A$ such that the equation 
$$
a_1 - a_2 = x :\,\,\, a_1, a_2 \in A
$$ 
has $\Omega(|A|^{1/2})$ solutions $(a_1, a_2)$ for at least $\Omega(|A|^{1/2})$ values of $x$.  

\end{remark}

\section{Acknowledgments}

The first author is supported by ERC-AdG. 321104 and Hungarian National Research Development and Innovation Funds  K 109789, NK 104183 and K 119528.

\noindent The second author is supported by the Knuth and Alice Wallenberg postdoctoral fellowship.

\noindent The work on this paper was partially carried out while the second author was visiting R\'enyi Institute of Mathematics by invitation of Endre Szemer\'edi, whose hospitality and support is greatly acknowledged. We also thank Peter Hegarty and Ilya Shkredov for useful discussions.

\bibliographystyle{plain}
\bibliography{convex_set_sumset}

\end{document}